\documentclass[12pt,reqno, a4wide]{amsart}
\usepackage{amssymb}
\usepackage{amsmath}
\usepackage{amsthm}
\usepackage{mathptmx}
\usepackage{color}
\usepackage[pdftex]{graphicx}
\usepackage{hyperref}
\usepackage{a4wide}
\usepackage{enumitem}
\usepackage{esint}

\newcommand{\R}{\mathbb R}

\newcommand{\Z}{\mathbb Z}
\newcommand{\T}{\mathbb T}

\newcommand{\N}{\mathbb N}

\newcommand{\p}{\partial}

\newcommand{\x}{\mathbf x}
\newcommand{\m}{\mathbf m}
\newcommand{\n}{\mathbf n}

\newtheorem{theorem}{Theorem}[section]

\newtheorem{remark}[theorem]{Remark}
\newtheorem{lemma}[theorem]{Lemma}

\numberwithin{equation}{section}
\begin{document}
\title[On well-posedness]{On the Periodic Zakharov-Kuznetsov  equation}

\author{F. Linares}
\address{IMPA\\
 Estrada Dona Castorina 110, Rio de Janeiro, 22460-320,  Brazil.}
\email{linares@impa.br}
\author{M. Panthee}
\address{IMECC-UNICAMP\\
13083-859, Campinas, S\~ao Paulo,  Brazil}
\email{mpanthee@ime.unicamp.br}
\thanks{MP was partially supported by FAPESP (2016/25864-6) and CNPq (305483/2014-5) Brasil.}
\author{T. Robert}
\address{
Universit\'e de Cergy-Pontoise,  Cergy-Pontoise, F-95000, UMR 8088 du CNRS
}
\email{tristan.robert@u-cergy.fr}
\author{N. Tzvetkov}
\address{
Universit\'e de Cergy-Pontoise,  Cergy-Pontoise, F-95000, UMR 8088 du CNRS
}
\email{nikolay.tzvetkov@u-cergy.fr}



\begin{abstract} We consider the Cauchy  problem  associated with the Zakharov-Kuznetsov equation, posed on $\T^2$.  We prove the local well-posedness for given data in $H^s(\T^2)$ whenever $s> 5/3$.   More importantly, we prove that  this equation is of quasi-linear type for initial data in any Sobolev space on the torus, 
in sharp contrast with its semi-linear character in the $\R^2$ and $\R\times \T$ settings.
\end{abstract}
\maketitle
\section{Introduction}
In this work we consider the initial value problem (IVP) associated with the following bi-dimensional dispersive model
\begin{equation}\label{zk-1}
\begin{cases}
\p_tw+\p_x^3w+\p_x\p_y^2w +w\p_xw=0, \hskip15pt \mathbf{x}=(x,y)\in \T^2, \;\;t\in\R,\\
w(0,x,y)=w_0(x,y),
\end{cases}
\end{equation}
where $w:\R\times\T^2\to \R$, $\T^2:=\R^2/(2\pi\Z)^2$.
\\

The model presented in \eqref{zk-1} is a bi-dimensional generalization  of the   Korteweg-de Vries (KdV) equation, and was introduced in \cite{ZK}   
to describe the propagation of nonlinear ion-acoustic waves in magnetized plasma (for its rigorous derivation we refer to \cite{LLS}). 
This model is widely known as the Zakharov-Kuznetsov (ZK) equation and is extensively studied in the literature, see for example \cite{Fa, G-H, LPS10, LP09, LS, M-P,RV} and the references therein. 
 \\
 
 If \eqref{zk-1} is posed on $\R^2$, it is known that it is locally well-posed in $H^s(\R^2)$, $s>1/2$ (see \cite{M-P,G-H}).
 Furthermore,  if \eqref{zk-1} is posed on $\R\times\T$ then it is well-posed in $H^s(\R\times\T)$, $s\geq 1$ (see \cite{M-P}).
 In all these results the dispersion in the $x$ direction plays a crucial role, the solutions are constructed by the Picard iteration and the flow map is smooth on $H^s$.
 \\
  
 Our interest here is  in studying  the local  well-posedness issue to the IVP \eqref{zk-1} for given data in the periodic Sobolev space  $H^s(\T^2)$, $s\in \R$. 
 We note that for $s>2$ one can ignore the dispersive effects and solve  \eqref{zk-1} as a quasi-linear hyperbolic equation leading to  the local well-posedness of  \eqref{zk-1} in $H^s(\T^2)$, $s>2$.   Our first goal is to prove a well-posedness result in spaces of lower regularity. 
 \begin{theorem}\label{local-zk}
Let $s>5/3$. Then for every  $w_0\in H^s(\T^2)$ there exist a time $T=T(\|w_0\|_{H^s})$ and a unique solution $w\in C([0,T]: H^s(\T^2))$ to the IVP \eqref{zk-1} such that $w, \partial_x w, \partial_y w \in L_T^1L_{xy}^{\infty}$. Moreover, the map that takes the initial data to the solution $w_0\mapsto w\in C([0,T]: H^s(\T^2))$ is continuous.
\end{theorem}
This local well-posedness result for the IVP associated with the ZK equation  \eqref{zk-1} posed in a purely periodic domain is the first one exploiting in a non trivial way the dispersive effect.  The method of proof of Theorem~\ref{local-zk} is nowadays standard (see e.g. \cite{BGT,IK,CK,KK, KTz} ). It uses in a crucial way short time Strichartz estimates for the linear part of the equation, see Lemma~\ref{lem-zk1}.  
Our short time Strichartz estimate is derived by purely local in space considerations in the spirit of \cite{BGT}. 
It may be that global in space considerations based on number theory arguments may improve the result of Lemma~\ref{lem-zk1}. 
To the best of our knowledge, the works \cite{KPV91,LV} are the first papers where short time Strichartz estimates on a compact spatial domain were considered.  
We also expect that further improvements can be obtained by using the ideas introduced in \cite{IKT}.
\\

Our main purpose in this work is to show that the flow map constructed in Theorem~\ref{local-zk} is not locally uniformly continuous, which highlights the quasilinear behaviour of the ZK equation (\ref{zk-1}) when considered with periodic boundary conditions. For previous related contributions, we refer to \cite{KochTzvetkov2003BO,KochTzvetkov2008,Robert2018}.
\begin{theorem}\label{lemma1.1} 
Let $s>5/3$. There exist two positive constants $c$ and $C$ and two sequences $(u_m)$ and $(v_m)$ of solutions to (\ref{zk-1}) in $C([0,1] : H^s(\T^2))$ such that
\begin{equation}\label{unif_bound}
\sup_{t\in[0,1]}\|u_m(t)\|_{H^s}+\|v_m(t)\|_{H^s} \leq C,
\end{equation}
and satisfying initially
\begin{equation}\label{init_bound}
\lim_{m\rightarrow +\infty}\|u_m(0)-v_m(0)\|_{H^s}=0,
\end{equation}
but such that for every $t\in[0,1]$,
\begin{equation}\label{failure_bound}
\liminf_{m\rightarrow +\infty}\|u_m(t)-v_m(t)\|_{H^s}\geq ct.
\end{equation}
\end{theorem}
As mentioned above the result of Theorem~\ref{lemma1.1} shows that the ZK equation behaves quite differently in the case of periodic boundary conditions in the $x$ variable since in the case when ZK is posed on $\R^2$ or $\R\times\T$ a result as Theorem~\ref{lemma1.1}  cannot hold true. 
As in the previous related results in the proof of Theorem~\ref{lemma1.1}  we exploit resonant interactions. 
We use the resonances coming from the zero  $x$ frequency.  
Let us recall that in the context of the KdV or the KP equations we can eliminate this interaction by the use of a simple gauge transform (see \emph{e.g.} \cite{JB_kp,Robert2017}). 
We are not aware of such a gauge transform in the context of the ZK equation and we suspect that the situation is more intricate compared to KdV or KP. 
It is however not excluded that one may reduce the ZK equation on the torus to a semi-linear problem by the use of a suitably chosen gauge transform, 
at least for some class of initial data. 
\\

Let us finally mention that since we work in the periodic setting, in the proof of  Theorem~\ref{lemma1.1} 
we cannot use the localization in space arguments as in \cite{KochTzvetkov2003BO,KochTzvetkov2008}.
This makes that in order to construct the approximation of the true solution, we exploit  a \lq\lq curvature property\rq\rq~of the resonance function associated with the ZK equation 
(see also (\ref{eq-curv-reson}) and Remark~\ref{FIN} below). 
\section{Preliminary estimates}\label{sec-3}

In this section we derive some preliminary estimates that are useful to prove the local well-posedness  result of this work. We begin with the preparation to obtain the Strichartz estimate in localized form.

For $N\in 2^{\N}\cup \{0\}$, we define a projection operator $P_N$ as a Fourier multiplier operator by
\begin{equation}\label{eq-d2}
\begin{cases}
\widehat{P_Nf}(\m) := \chi_{[N;2N[}(|\m|)\widehat{f}(\m),~N\geq 1\\
\widehat{P_0f}(\m) := \chi_{[0;1[}(|\m|)\widehat{f}(\m),~N=0,
\end{cases}
\end{equation}
where $\m=(m_1,m_2)\in\Z^2$, and $\chi_I$ is the indicator function of the interval $I$. In terms of the projection operator, we can write an equivalence for the Sobolev norms on the torus in the following way
\begin{equation*}\label{eq-n2}
\|f\|_{H^s(\T^2)}^2 \sim \sum_{N\geq 0} (1\vee N)^{2s}\|P_Nf\|_{L^2(\T^2)}^2,
\end{equation*}
where we use the notation $a\vee b = \max(a,b)$.

Consider now the linear problem
\begin{equation}\label{zk-6}
\begin{cases}
\p_tw+\p_{x}^3w +\p_x\p_{y}^2w =0, \quad (x,y)\in \T^2,\; t\in \R,\\
w(0,x, y) = w_0(x,y),
\end{cases}
\end{equation}
and let $W(t)$ defined by 
 \begin{equation}\label{w-t}
 W(t)w_0(\x) =\frac{1}{(2\pi)^2} \sum_{\m\in\Z^2}e^{i(\m\cdot\x+(m_1^3+m_1m_2^2)t)}\widehat{w_0}(\m),
 \end{equation}
 be the unitary group that describes the solution to  \eqref{zk-6}.  

In what follows we prove the localized version of the Strichartz estimate satisfied by the unitary group  $W(t)$.

\begin{lemma}\label{lem-zk1} 
Let $W(t)$ be as defined in \eqref{w-t} and $P_N$ be the operator defined in \eqref{eq-d2}. Then for any $w_0\in L^2(\T^2)$ and time interval $I\subset \R$ with $|I|\sim (1\vee N)^{-2}$, 
\begin{equation}\label{eq-zk2a}
\|W(t)P_Nw_0\|_{L^{2}_{I} L_{xy}^{\infty}} \lesssim (1\vee N)^{-1/3} \, \|w_0\|_{L^2_{xy}},
\end{equation}
and
\begin{equation}\label{eq-l3zk}
\|W(t) w_0\|_{L^{2}_{ [0, 1]} L_{xy}^{\infty}} \lesssim \, \|w_0\|_{H^{\frac23+}}.
\end{equation}
\end{lemma}

\begin{proof}
We start by proving \eqref{eq-zk2a}. Since it is straightforward for $N=0$, we assume $N\geq 1$ in the following. For $\x =(x,y)$, using the definition of the group $W(t)$ and the projection operator $P_N$, we have that
\begin{equation*}\label{eq-p1}
W(t)P_Nw_0 = \frac{1}{(2\pi)^2 }\sum_{\m\in \Z^2}e^{i(\m\cdot\x+(m_1^3+m_1m_2^2)t)}\chi_{[N;2N[}(|\m|)\widehat{w_0}(\m).
\end{equation*}
In order to decouple the frequencies in the time oscillation, we will use the symmetrization argument of \cite{G-H}. First, to have good localizations in frequency after symmetrizing, we observe that if $|\m|< 2N$ then $|m_1\pm\frac{m_2}{\sqrt{3}}|< 4N$. Since $P_N$ is a bounded operator with norm one, it suffices to prove that (\ref{eq-zk2a}) holds with $P_N$ replaced with
\[\widetilde{P_N}w_0 = \sum_{\m\in\Z^2}\widetilde{\phi_{4N}}(m_1+\frac{m_2}{\sqrt{3}})\widetilde{\phi_{4N}}(m_1-\frac{m_2}{\sqrt{3}})\widehat{w_0}(m_1,m_2)e^{i\m\cdot\x},\]
where $\widetilde{\phi_{4N}}$ is a smooth cut-off defined as follows : let $\phi\in C_0^{\infty} (-2,2)$ be such that $\phi(r)=1$ for $|r|\leq 1$. Then write
\[\widetilde{\phi_N}(x) = \phi\left(\frac{x}{N}\right).\]
Indeed, with this definition and the remark above, we have $P_N = \widetilde{P_N}P_N$.


We can now define the operator
\begin{equation*}
\mathbf{T} = \chi_I(t)W(t)\widetilde{P_N}
\end{equation*}
where $\chi_I$ stands for the indicator function of the interval $I$.

To prove that $\mathbf{T}$ is bounded from $L^2(\T^2)$ to $L^{2}(I : L^{\infty}(\T^2))$ with norm less than $N^{-\frac13}$, the classical~\lq\lq$TT^*$\rq\rq~argument reduces the problem to show that
\begin{equation*}
g\in L^{2}(I:L^1(\T^2))\mapsto \mathbf{T}\mathbf{T}^*(g)(t)=\int_{\R}\chi_I(t)\chi_I(t')W(t-t')\widetilde{P_N}^2g(t')dt'\in L^{2}(I:L^{\infty}(\T^2))
\end{equation*}
is bounded with norm less than $N^{-\frac23}$. This last operator can be written as an integral operator, whose kernel is given by 
 \begin{equation*}\label{eq-p3.2}
 K_N(t, t', \x,\x'):= \chi_{I}(t)\chi_{I}(t')\sum_{\m\in\Z^2}\widetilde{\phi_{4N}}(m_1+\frac{m_2}{\sqrt{3}})^2\widetilde{\phi_{4N}}(m_1-\frac{m_2}{\sqrt{3}})^2 e^{i[\m\cdot( \x-\x')+(m_1^3+m_1m_2^2)(t-t')]}.
 \end{equation*}
 Note that the time localizations imply that $|t-t'|\lesssim N^{-2}$. Therefore the whole matter reduces to proving the following pointwise estimate on the kernel
 \begin{equation}\label{eq-p4}
 \Big|\sum_{\m\in\Z^2}\widetilde{\phi_{2N}}(m_1+\frac{m_2}{\sqrt{3}})^2\widetilde{\phi_{2N}}(m_1-\frac{m_2}{\sqrt{3}})^2e^{i(\m\cdot \x+(m_1^3+ m_1m_2^2)t)}\Big|\lesssim |t|^{-2/3},
 \end{equation}
 for any $\x\in \T^2$ and $N^{-3}\lesssim |t|\lesssim N^{-2}$.
 
 Indeed, with (\ref{eq-p4}) at hand, the bound on $\mathbf{T}\mathbf{T}^*$ follows from Young's inequality and the fact that
 \[\||t|^{-2/3}\chi_{[N^{-3};N^{-2}]}(t)\|_{L^{1}}\lesssim N^{-\frac23},\]
 whereas the contribution of $|t-t'|\lesssim N^{-3}$ is estimated roughly with the trivial bound $|K_N|\lesssim N^2$.

 To obtain (\ref{eq-p4}), we use Poisson summation formula \cite{SW}
 \begin{equation}\label{eq-p5}
 \sum_{\m\in\Z^2}F(\m) = \sum_{\n\in\Z^2}\widehat{F}(2\pi\n), \qquad \forall\; F\in \mathcal{S}(\R^2).
 \end{equation}
 
 Using \eqref{eq-p5}, the sum in the LHS of \eqref{eq-p4}  becomes
 \begin{equation*}
 \sum_{\n\in \Z^2}\int_{\R^2}e^{-i2\pi\n\cdot\xi} \widetilde{\phi_{4N}}(\xi_1+\frac{\xi_2}{\sqrt{3}})^2\widetilde{\phi_{4N}}(\xi_1-\frac{\xi_2}{\sqrt{3}})^2 e^{i(\x\cdot\xi+(\xi_1^3+\xi_1\xi_2^2t)}d\xi.
 \end{equation*}

 Here we symmetrize the linear evolution : by using the linear transformation in \cite[Section 2.1]{G-H} we can write the integrals within the sum above as
 \[c\int_{\R^2} \widetilde{\phi_{4N}}(\eta_1)^2\widetilde{\phi_{4N}}(\eta_2)^2 e^{i(\eta_1\frac{(x-2\pi n_1)+\sqrt{3}(y-2\pi n_2)}{2}+\eta_2\frac{(x-2\pi n_1)-\sqrt{3}(y-2\pi n_2)}{2}+\frac12(\eta_1^3+\eta_2^3)t}d\eta.\]

Thus (\ref{eq-p4}) can be expressed as
\begin{equation}\label{eq-p6}
\Big|\sum_{\n\in\Z^2} F_N\left(\frac{x+\sqrt{3}y-2\pi (n_1 +\sqrt{3}n_2)}{2}\right)F_N\left(\frac{x-\sqrt{3}y-2\pi(n_1-\sqrt{3}n_2)}{2}\right)\Big|\lesssim |t|^{-2/3},
\end{equation}
where
\[F_N(X)=\int_{\R} \widetilde{\phi_{4N}}(\xi)^2e^{i(\xi X+\frac12 \xi^3t)}d\xi.\]
First, note that from the time and frequency localizations, we have $|\xi^2 t|\lesssim 1$, thus for $|X|$ large enough the phase in the above integral has no stationary point. As a result, we can bound $|F_N(X)|$ for $|X|> C$ for some (large) constant $C>0$ by successive integrations by parts. Indeed, by writing $F_N$ as an abstract oscillatory integral ${\displaystyle \int_{\R}\psi(\xi)e^{i\Phi(\xi)}d\xi}$, using the definition of $\widetilde{\phi_{4N}}$ and the localizations in $\xi$ and $t$ we get that for any $k\in\N$ we have $\|\psi^{(k)}\|_{L^1}\lesssim N^{1-k}$ and $\|\Phi'\|_{L^{\infty}}\sim |X|$, $\|\Phi''\|_{L^{\infty}}\lesssim N^{-1}$ and $\|\Phi^{(3)}\|_{L^{\infty}}\lesssim N^{-2}$. Thus for $|X|>C$,
\[\int_{\R} \psi(\xi)e^{i\Phi(\xi)}d\xi =- \int_{\R}\left[(i\Phi')^{-1}\left\{(i\Phi')^{-1}\left((i\Phi')^{-1}\psi\right)'\right\}'\right]'e^{i\Phi(\xi)}d\xi=O(N^{-2}|X|^{-3}).\]
Now let us come back to (\ref{eq-p6}). Let $M>0$ be some large constant. Using that $x$ and $y$ lie in a compact set, if $|n_1+\sqrt{3}n_2|\vee |n_1-\sqrt{3}n_2|>M$ then we can use the previous bound for the corresponding term in the sum of (\ref{eq-p6}) along with the rough bound $\|F_N\|_{L^{\infty}}\lesssim N$ for the second term, to get
\begin{multline*}
\Big|\sum_{|n_1+\sqrt{3}n_2|\vee |n_1-\sqrt{3}n_2|>M} F_N\left(\frac{x+\sqrt{3}y-2\pi (n_1 +\sqrt{3}n_2)}{2}\right)F_N\left(\frac{x-\sqrt{3}y-2\pi(n_1-\sqrt{3}n_2)}{2}\right)\Big|\\
\lesssim \sum_{|n_1+\sqrt{3}n_2|\vee |n_1-\sqrt{3}n_2|>M} N^{-1}(|n_1+\sqrt{3}n_2|\vee |n_1-\sqrt{3}n_2|)^{-3}\\
\lesssim N^{-1}\sum_{|n_1+\sqrt{3}n_2|\vee |n_1-\sqrt{3}n_2|>M} [(1\vee|n_1|)(1\vee|n_2|)]^{-3/2}.
\end{multline*}
In view of the last bound, to prove (\ref{eq-p6}) it remains to treat the contribution of the sum for $|n_1+\sqrt{3}n_2|<M$ and $|n_1-\sqrt{3}n_2|<M$. In this case we have in particular $|\n|\leq 2M$ so that the sum is finite (uniformly in $N$), thus it is enough to get the bound 
\begin{equation}\label{eq-p6bis}
\|F_N\|_{L^{\infty}}\lesssim |t|^{-1/3}.
\end{equation}
We first write $F_N$ as
 \[F_N(X)=\int_{\R}e^{i(X\xi+\frac{\xi^3}{2}t)}d\xi-\int_{\R}\left[1-\widetilde{\phi_{4N}}^2(\xi)\right]e^{i(X\xi+\frac{\xi^3}{2}t)}d\xi=I+II.\]
 The first term is
 \[I = \frac{c}{t^{\frac13}}\int_{\R} e^{i(\sqrt[3]{2}Xt^{-1/3}\eta+\eta^3)} d\eta.\]
 From  \cite{KPV91} we have that the last integral is bounded (see also Lemma 7.2 in \cite{LP-15} or Lemma 3.6 in \cite{KPV-93}), so it remains to estimate the second term. We observe that the phase $\Phi(\xi)=X\xi+\frac{\xi^3}{2}t$ satisfies $|\Phi''(\xi)|= 3|t\xi|\gtrsim N|t|$ on the support of $(1-\widetilde{\phi_{4N}}^2)$. Thus we can conclude from van der Corput lemma (see \emph{e.g.} \cite[Chapter 8]{Stein}) that
\[|II|\lesssim (N|t|)^{-\frac12}\left\{\|1-\widetilde{\phi_{4N}}^2\|_{L^{\infty}}+\|(1-\widetilde{\phi_{4N}}^2)'\|_{L^1}\right\}\lesssim |t|^{-\frac13},\]
where in the last step we have used the lower bound on $|t|$. This proves (\ref{eq-p6bis}).

The  proof of the estimate \eqref{eq-l3zk} follows from {\it the short time Strichartz estimate} using a  Littlewood-Paley procedure. In fact, splitting the interval $[0, 1]$ on $(1\vee N)^2$ intervals of size $(1\vee N)^{-2}$ we obtain, using {\it the short time Strichartz estimate} (and that $P_N^2=P_N$)
\begin{equation*}\label{eq-p10}
 \|W(t)P_Nw_0\|_{L^{2}([0, 1]; L_{xy}^{\infty})}^{2} 
 \leq c (1\vee N)^2(1\vee N)^{-\frac23}\|P_Nw_0\|_{L_{xy}^2}^{2}.
 \end{equation*}
 Summing the last estimate over $N\geq 0$ and using the equivalence of norms, we conclude that
  \begin{equation*}\label{eq-p12}
 \|W(t)w_0\|_{L^{2}([0, 1]; L_{xy}^{\infty})}
 \leq c \|w_0\|_{H^{\frac23+}}.
 \end{equation*}
 This completes the proof of Lemma~\ref{lem-zk1}.
 \end{proof}
Next, we move to prove  a version of the Kato-Ponce  Commutator  estimate in $\mathbb T^2$, 
which is an extension  of the one proved for the $\T$-case in \cite{IK}.

First, let us define \[ H^{\infty}(\mathbb T^2)=\bigcap_{s\in\R}H^{s}(\T^2).\] For $s\in\R$ we also define the Fourier multiplier \[\widehat{J^sf}(\m):=\widehat{J^s_{\T^2}f}(\m) = (1+|\m|^2)^{\frac{s}2}\widehat{f}(\m).\]
\begin{lemma}\label{lem2} Let $s\geq 1$ and $ f, g \in H^{\infty}(\mathbb T^2)$. 
Then
\begin{equation}\label{eq2.5}
\begin{split}
\|J^s(fg)-fJ^sg&\|_{L^2(\mathbb T^2)}\le c\,\Big\{\|J^sf\|_{L^2(\mathbb T^2)}\|g\|_{L^{\infty}(\mathbb T^2)}\\
&+
(\|f\|_{L^{\infty} (\mathbb T^2)}+\|\nabla f\|_{L^{\infty} (\mathbb T^2)})\,\|J^{s-1}g\|_{L^2(\mathbb T^2)}\Big\}.
\end{split}
\end{equation}
\end{lemma}
Lemma~\ref{lem2} is a bidimensional generalization of \cite[Lemma 9.A.1 (b)]{IK}, it follows from the Kato-Ponce commutator estimate on $\R^2$ \cite{KP} through the same argument as in \cite{IK}.

\section{{\em A PRIORI} estimates: proof of the Theorem \ref{local-zk}}\label{sec-4}

In this section we use the estimates established in the previous section to establish some more estimates and prove the  main result regarding local well-posedness.  We begin by proving an {\em a priori} estimate that plays a fundamental role in our argument. 
\begin{lemma}\label{lem2-zk}
Let $w_0\in H^{\infty}(\T^2)$ and $w$ be the corresponding smooth solution to the IVP \eqref{zk-1}. Then for any $T\in [0, 1]$ and $s\geq 1$, we have
\begin{equation}\label{ap-1zk}
\|w\|_{L_T^{\infty}H^s(\T^2)}\leq c_s\exp\big(c_s(\|w\|_{L_T^1L_{xy}^{\infty}}+\|\nabla w\|_{L_T^1L_{xy}^{\infty}})\big)\|w_0\|_{H^s(\T^2)}.
\end{equation}
\end{lemma}
\begin{proof}
We  apply the operator $J^s$ to \eqref{zk-1}, multiply the resulting equation by $J^sw$ and then integrate by parts, to obtain
\begin{equation}\label{est-2zk}
\frac12\frac{d}{dt}\int_{\T^2}(J^sw)^2\,dxdy +\int_{\T^2}J^s(w\partial_xw)J^sw\, dxdy=0.
\end{equation}

Using the commutator notation $[A, B]f = A(Bf) -B(Af)$,  integration by parts and the Cauchy-Schwarz inequality, we obtain from \eqref{est-2zk} that
\begin{equation}\label{eq2.8}
\begin{split}
\frac12\frac{d}{dt}\int_{\T^2}(J^sw)^2dxdy  &= -\int_{\T^2}\{J^s(\partial_xw)-wJ^s\partial_xww +wJ^s\partial_xw\}J^sw\,dxdy\\
&= -\int_{\T^2}([J^s, w]\partial_xw)J^su\, dxdy+\frac12\int_{\T^2}\partial_x w(J^sw)^2 \,dxdy\\
&\leq \|[J^s, w]\partial_xw\|_{L^2(\T^2)} \|J^sw\|_{L^2(\T^2)} +\frac12\|\partial_xw\|_{L^{\infty}(\T^2)}\|J^sw\|_{L^2(\T^2)}^2.
\end{split}
\end{equation}

In the light of estimate \eqref{eq2.5} in Lemma \ref{lem2}, we have
\begin{equation}\label{eq2.9}
\begin{split}
\|[J^s, w]\partial_xw\|_{L^2(\T^2)}\leq c&\big\{\|J^sw\|_{L^2(\T^2)}\|\partial_x w\|_{L^{\infty}(\T^2)}\\ 
&+(\|w\|_{L^{\infty}(\T^2)}+\|\nabla w\|_{L^{\infty}(\T^2)})\|J^{s-1}\partial_xw\|_{L^2(\T^2)}\big\}.
\end{split}
\end{equation}

Inserting \eqref{eq2.9} in \eqref{eq2.8}, we obtain, after simplification 
\begin{equation}\label{eq2.10}
\frac{d}{dt}\|w(t)\|_{H^s(\T^2)}^2 \leq c\big(\|w\|_{L^{\infty}(\T^2)} + \|\nabla w\|_{L^{\infty}(\T^2)}\big)\|w(t)\|_{H^s(\T^2)}^2.
\end{equation}

Using Gronwall's inequality, \eqref{eq2.10} yields
\begin{equation*}
\|w\|_{L_T^{\infty}H^s(\T^2)}^2\leq \|w_0\|_{H^s(\T^2)}^2\exp\Big(c\int_0^T\big(\|w\|_{L^{\infty}(\T^2)} + \|\nabla w\|_{L^{\infty}(\T^2)}\big)dt\Big),
\end{equation*}
which gives the required estimate \eqref{ap-1zk}.
\end{proof}

Following the approach in \cite{CK} we obtain the following  estimate that will be useful in our argument.

\begin{lemma}\label{lem3-zk}  Let $F\in L^1([0,T]:L^2(\T^2))$. Then any solution of the equation
\begin{equation}\label{eqn-12}
\partial_t w + \partial_x^3w+\partial_x\partial_y^2w+\partial_x F(w)=0, \hskip15pt (x,y)\in\T^2, \;\;t\in\R,
\end{equation}
satisfies
\begin{equation}\label{ineq-zk}
\|w\|_{L^1_TL^{\infty}_{xy}}\lesssim T^{\frac12}\,\Big(\|w\|_{L^{\infty}_T H^{\frac23+}}+\|F\|_{L^1_T L^2_{xy}}\Big).
\end{equation}
\end{lemma}

\begin{proof} 
Let $P_N$ be
the projection operator defined in (\ref{eq-d2}), and fix such a number $N\in 2^{\N}\cup\{0\}$. We divide the interval $[0;T]$ in subintervals $[a_k, a_{k+1})$ of size $T(1\vee N)^{-2}$ for $k=1, \dots, (1\vee N)^2$. Then

\begin{equation}\label{ineq1-zk}
\begin{split}
\|P_N w\|_{L^1_T L^{\infty}_{xy}} &\le \underset{k=1}{\overset{(1\vee N)^2}{\sum}} \|\chi_{[a_k, a_{k+1})}(t) P_N w\|_{L^1_T L^{\infty}_{xy}}\\
&\lesssim \underset{k=1}{\overset{(1\vee N)^2}{\sum}} \big(\int_{a_k}^{a_{k+1}}\,dt\big)^{1/2}  \|\chi_{[a_k, a_{k+1})}(t) P_N w\|_{L^{2}_T L^{\infty}_{xy}}\\
&\lesssim  (T(1\vee N)^{-2})^{1/2}  \underset{k=1}{\overset{(1\vee N)^2}{\sum}}  \|\chi_{[a_k, a_{k+1})}(t) P_N w\|_{L^{2}_T L^{\infty}_{xy}}.
\end{split}
\end{equation}

On the other hand, using Duhamel's formula, for $t\in [a_k, a_{k+1})$, we obtain from \eqref{eqn-12}
\begin{equation}\label{ineq2-zk}
w(t)= W(t-a_k)\, w(a_k) -\int_{a_k}^t  W(t-s) [\partial_x F(w)]\, ds.
\end{equation}

Using Lemma \ref{lem-zk1} it follows from \eqref{ineq2-zk} that 
\begin{equation}\label{ineq3-zk}
\begin{split}
 \|\chi_{[a_k, a_{k+1})}(t) P_N w\|_{L^{2}_T L^{\infty}_{xy}} \lesssim (1\vee N)^{-\frac13} \|P_N w(a_k)\|_{L^2_{xy}}
 + (1\vee N)^{-\frac13}  N \|\chi_{[a_k, a_{k+1})}(t)  P_N F\|_{L^1_T L^2_{xy}}.
\end{split}
 \end{equation}
 
 Combining \eqref{ineq1-zk} and \eqref{ineq3-zk}, we have
 \begin{equation}\label{ineq4-zk-s}
 \begin{split}
 \|P_N w\|_{L^1_T L^{\infty}_{xy}} &\lesssim (T(1\vee N)^{-2})^{1/2}  \underset{k=1}{\overset{(1\vee N)^2}{\sum}} \Big((1\vee N)^{-\frac13} \|P_N w(a_k)\|_{L^2_{xy}}\Big ) \\
&\hskip15pt +\,(T(1\vee N)^{-2})^{1/2}  \underset{k=1}{\overset{(1\vee N)^2}{\sum}} (1\vee N)^{-\frac13} N \|\chi_{[a_k, a_{k+1})}(t)  P_N F\|_{L^1_T L^2_{xy}}\\
&\lesssim (T(1\vee N)^{-2})^{1/2} (1\vee N)^{-\frac13} \underset{k=1}{\overset{(1\vee N)^2}{\sum}}  \|P_N w(a_k)\|_{L^2_{xy}}\\
 &+(T(1\vee N)^{-2})^{1/2}(1\vee N)^{-\frac13}  N  \|P_N F\|_{L^1_T L^2_{xy}}.
\end{split}
 \end{equation}

Simplifying further, one obtains from \eqref{ineq4-zk-s} that
\begin{equation}\label{ineq4-zk}
 \begin{split}
\|P_N w\|_{L^1_T L^{\infty}_{xy}} &\lesssim   T^{\frac12} \Big ( (1\vee N)^{\frac23} (1\vee N)^{-2}  \underset{k=1}{\overset{(1\vee N)^2}{\sum}}  \|P_N w(a_k)\|_{L^2_{xy}}
+(1\vee N)^{-\frac13}  \|P_N F\|_{L^1_T L^2_{xy}}\Big)\\
&\lesssim  T^{\frac12} (1\vee N)^{-\epsilon} \Big ( (1\vee N)^{-2}  \underset{k=1}{\overset{(1\vee N)^2}{\sum}} 
 \| (1\vee N)^{\frac23+\epsilon} P_N w(a_k)\|_{L^2_{xy}}
+\|P_N F\|_{L^1_T L^2_{xy}}\Big)
 \end{split}
 \end{equation}
for $0<\epsilon < \frac13$.

Thus inequality \eqref{ineq-zk} follows from summing \eqref{ineq4-zk} over $N$.
\end{proof}

 Now, we record a  key estimate to prove the local well-posedness result for the IVP \eqref{zk-1}.

\begin{lemma}\label{lem4-zk} Let $w$ be a solution of \eqref{zk-1} with $w_0\in H^{\infty}(\T^2)$ defined in $[0,T]$.
Then for any $s>5/3$, there exist  $T=T(\|w_0\|_{H^s})$ and a constant $c_T(\|w_0\|_{H^s}, s)$ such that
\begin{equation}\label{apriori-zk}
g(T):=\int_0^T \big(\|w(t)\|_{L^{\infty}_{xy}}+\|\nabla w\|_{L^{\infty}(\T^2)}\big)\,dt\le c_T.
\end{equation}
\end{lemma}

\begin{proof}
Observe that $w$, $\partial_xw$ and $\partial_y w$  satisfy Lemma \ref{lem3-zk} with $F$ given, respectively, by $\frac12w^2$,  $\frac12\partial_x(w^2)$ and $\frac12\partial_y(w^2)$.
Hence,  for $s'>  2/3$, using \eqref{ineq-zk} for $w$, $\partial_x w$ and $\partial_y w$ with respective $F$, we obtain
\begin{equation}\label{ineq5-zk}
\begin{split}
g(T) &\le c\, T^{\frac12}\,\Big( \|J^{s'} w\|_{L^{\infty}_TL^2_{xy}}+ \|J^{s'} \partial_x w\|_{L^{\infty}_TL^2_{xy}}+ \|J^{s'} \partial_y w\|_{L^{\infty}_TL^2_{xy}}\\
&\hskip10pt + \int_0^T  \| w^2\|_{L^2_{xy}}\,dt+  \int_0^T\|\partial_x (w^2)\|_{L^2_{xy}}\,dt+  \int_0^T\|\partial_y (w^2)\|_{L^2_{xy}}\,dt\Big).
\end{split}
\end{equation}

It follows that, for $s>5/3$,
\begin{equation}\label{ineq6-zk}
\|J^{s'} w\|_{L^{\infty}_TL^2_{xy}}+ \|J^{s'} \partial_x w\|_{L^{\infty}_TL^2_{xy}}+ \|J^{s'} \partial_y w\|_{L^{\infty}_TL^2_{xy}} \le c\, \|J^s w\|_{L^{\infty}_TL^2_{xy}}.
 \end{equation}

On the other hand,
\begin{equation}\label{ineq7-zk}
\begin{split}
 \int_0^T &( \| w^2\|_{L^2_{xy}}+\|\partial_x (w^2)\|_{L^2_{xy}}+\|\partial_y (w^2)\|_{L^2_{xy}})\,dt\\
 &\le c\, \;\|w\|_{L^{\infty}_TL^2_{xy}}\int_0^T (\|w(t)\|_{L^{\infty}_{xy}}+\|\partial_x w(t)\|_{L^{\infty}_{xy}}+\|\partial_y w(t)\|_{L^{\infty}_{xy}}\big)\,dt\\
 &\le c\, \;\|w\|_{L^{\infty}_T{H^s}}\int_0^T (\|w(t)\|_{L^{\infty}_{xy}}+\|\nabla w\|_{L^{\infty}_{xy}}\big)\,dt.
 \end{split}
 \end{equation}
 
 Combining \eqref{ineq5-zk}, \eqref{ineq6-zk}, \eqref{ineq7-zk} with Lemma \ref{lem2-zk} we obtain the inequality
 \begin{equation}\label{apriori2-zk}
 g(T) \le cT^{\frac12}\,\|w_0\|_{H^s}\,\exp{(c\, g(T))}(1+g(T)).
 \end{equation}
 
 From \eqref{apriori2-zk}  we can deduce the required result analogously as in \cite{CK}, so we omit the details. 
 \end{proof}

\begin{proof}[Proof of Theorem \ref{local-zk}]
The main tools in the proof are the results from Lemmas \ref{lem2-zk} and \ref{lem4-zk}. 

Let $s>5/3$ and consider an initial datum $w_0\in H^s(\T^2)$. As already mentioned, for $s>2$, we can prove the local well-posedness in $H^s$ of (1.1) by solving it as a quasi-linear hyperbolic PDE. So, we may consider $w_0\in H^s$, $5/3<s\leq 2$. Density of $H^{\infty}$ in $H^s$ allows one to find $w_0^{\epsilon}\in H^{\infty}$ such that $\|w_0^{\epsilon}-w_0\|_{H^s}\to 0$. Moreover, one has $\|w_0^{\epsilon}\|_{H^s}\leq c \|w_0\|_{H^s}$.

For $0<\epsilon<1$, let $w^{\epsilon}$ be the solution to the  IVP \eqref{zk-1} corresponding to the initial data $w_0^{\epsilon}\in H^{\infty}$ on $[0, \tau]$, $\tau>0$.  
One can use Lemma \ref{lem4-zk} to extend $w^{\epsilon}$ on a time interval $[0,T]$,   $T=T(\|w_0\|_{H^s} )>0$ and  to show that there exists a constant $c_T$ such that
\begin{equation}\label{pf-3}
\int_0^T \big(\|w^{\epsilon}(t)\|_{L^{\infty}_{xy}}+\|\partial_x w^{\epsilon}(t)\|_{L^{\infty}_{xy}}+\|\partial_y w^{\epsilon}(t)\|_{L^{\infty}_{xy}}\big)\,dt\le c_T.
\end{equation}
Also, from Lemma \ref{lem2-zk}, one has
\begin{equation}\label{pf-4}
\sup_{0\leq t \leq T}\|w^{\epsilon}\|_{H^s(\T^2)}\leq c_T.
\end{equation}

Now, using Gronwall's inequality and the estimate \eqref{pf-3}, one can show that
\begin{equation}\label{pf-5}
\sup_{0\leq t \leq T}\|w^{\epsilon}-w^{{\epsilon}'}\|_{L^2(\T^2)}\to 0\qquad {\rm{as}}\quad \epsilon,\,\epsilon '\to 0.
\end{equation}

In view of \eqref{pf-5} and \eqref{pf-4}, one can get, for $s'<s$, $w\in C([0, T];H^{s'})\cap L^{\infty}([0, T]; H^s)$ such that $w^{\epsilon}\to w$ in $C([0, T]; H^{s'})$. Indeed,  \eqref{pf-5}  implies that as $\epsilon\to 0$, $w^{\epsilon}\to w$ in $C([0, T], L^2)$. In the light of estimate \eqref{pf-4} we note that, $w^{\epsilon}\in L^{\infty}([0, T], H^s)$. Hence, by weak* compactness, $w\in L^{\infty}([0, T], H^s)$.

Once again, one can use Gronwall's inequality,  to prove that $w$ is the unique solution to the IVP \eqref{zk-1}. An usual Bona--Smith \cite{BS} argument can be used to prove the continuity of the solution $w(t)$ and the continuity of the flow-map in $H^s$. For a detailed exposition of this argument, we refer to a recent work \cite{KP16}.
\end{proof}



\section{Failure of the local uniform continuity of the flow map }\label{sec-5}

In this section we prove the failure of the local uniform continuity of the flow-map for the IVP  \eqref{zk-1} associated to the ZK equation. Let $\varphi(m, n) = m^3+mn^2$ be the time frequency of the solution to the linear problem associated to \eqref{zk-1} (see also \eqref{w-t}) with spatial frequency $\m=(m,n)$.
 
Following the idea in \cite{KochTzvetkov2003BO,KochTzvetkov2008,Robert2018}, we will  construct a family of approximate solutions whose frequencies lie in the region where the resonance relation 
\begin{equation}\label{resonance}
R(m,m_1,n,n_1):=\varphi(m, n)-\varphi(m-m_1, n-n_1)-\varphi(m_1, n_1),
\end{equation}
vanishes. Recall that the resonance relation \eqref{resonance} can explicitly be written  as 
\begin{equation}\label{reso}
\begin{split}
R(m,m_1,n,n_1)&:= 3mm_1(m-m_1)+\big[mn^2-m_1n_1^2 -(m-m_1)(n-n_1)^2\big]\\
& = 3mm_1(m-m_1) +2nn_1(m-m_1) +m_1n^2-mn_1^2.
\end{split}
\end{equation}
In particular, note that for any choice of $(m,n)\in \Z^2$, we have $R(m,0,n,2n)=0$. Consequently, we will exploit this resonant interaction to construct the solutions in 
Theorem~\ref{lemma1.1}.

\begin{proof}[Proof of Theorem~\ref{lemma1.1}]
Let us fix $s>5/3$. For $\theta \in [-1,1]$ and $m\in\N^*$, let us define the family of functions on $[0,1]\times\T^2$ by
\begin{multline}\label{utheta}
u_{\theta,m}(t,x,y):= \theta m^{-1}\cos(2y) + \cos\left(\frac{\theta}{2}t\right)m^{-s}\cos\left(mx-y+\varphi(m,-1)t\right)\\+\sin\left(\frac{\theta}{2}t\right)m^{-s}\sin\left(mx+y+\varphi(m,1)t\right)+r_{\theta,m}(t,x,y).
\end{multline}
Note that all three modes above are solutions to the linear part of (\ref{zk-1}) modulated by a time oscillating factor, and the third one corresponds to the main part of the nonlinear interaction of the first two (see below). The last term is given by
\begin{multline}\label{remainder}
r_{\theta,m} =m^{-s}R(m,0,-1,2)^{-1}\cos\left(\frac{\theta}{2}t\right)\cos\left(mx-3y+\varphi(m,-1)t\right)\\ + m^{-s}R(m,0,1,-2)^{-1}\sin\left(\frac{\theta}{2}t\right)\sin\left(mx+3y+\varphi(m,1)t\right).
\end{multline}
It allows to cancel the remaining interactions (due to the requirement of working with non localized real-valued solutions) in order for these approximate solutions to satisfy the equation up to a sufficiently small error. Indeed, we have the estimate
\begin{equation}\label{approx-sol}
\|(\p_t+\p_x\Delta)u_{\theta,m}+u_{\theta,m}\p_xu_{\theta,m}\|_{H^s}\lesssim m^{(-1-s)\vee (1-2s)},
\end{equation}
which holds uniformly in $(\theta,t)\in[-1,1]\times[0,1]$ (and $a\vee b$ stands for the maximum between $a$ and $b$).

To prove (\ref{approx-sol}), let us write $u_i$, $i=1,2,3$ for the first three terms in the definition of $u_{\theta,m}$ (\ref{utheta}), then the term inside the norm on the left-hand side of (\ref{approx-sol}) reads
\[(\p_t+\p_x\Delta)(u_2+u_3+r_{\theta,m}) + u_1\p_x(u_2+u_3) +\widetilde{r},\]
with $\widetilde{r}=u_1\p_x r_{\theta,m} + (u_2+u_3+r_{\theta,m})\p_x u_{\theta,m}$.

A key ingredient here is that 
\begin{equation}\label{eq-curv-reson}
R(m,0,-1,2)=R(m,0,1,-2)=-8m,
\end{equation}
so that $\widetilde{r}=O_{L^2}(m^{-1-s}+m^{1-2s})$, where we use the notation $O_{L^2}(m^{\alpha})$ for functions having an $L^2(\T^2)$-norm bounded by a constant times $m^{\alpha}$, uniformly in $(\theta,t)\in[-1,1]\times[0,1]$.

Next, straightforward computations give
\begin{multline}\label{comput1}
(\p_t+\p_x\Delta)(u_2+u_3) = - \frac{\theta}{2}\sin\left(\frac{\theta}{2}t\right)m^{-s}\cos\left(mx-y+\varphi(m,-1)t\right)\\ + \frac{\theta}{2}\cos\left(\frac{\theta}{2}t\right)m^{-s}\sin\left(mx+y+\varphi(m,1)t\right),
\end{multline}
and
\begin{multline}\label{comput2}
u_1\p_x(u_2+u_3)= -\frac{\theta}{2}\cos(\frac{\theta}{2}t)m^{-s}\left[\sin(mx+y+\varphi(m,-1)t)+\sin(mx-3y+\varphi(m,-1)t)\right]\\
+\frac{\theta}{2}\sin(\frac{\theta}{2}t)m^{-s}\left[\cos(mx-y+\varphi(m,1)t)+\cos(mx+3y+\varphi(m,1)t)\right].
\end{multline}
Since $\varphi(m,-1)=\varphi(m,1)$ (as the symbol is even in $n$), we see that the first and third terms in the nonlinear interaction (\ref{comput2}) are counterbalanced by the linear evolution (\ref{comput1}). To cancel the two remaining terms, we finally compute
\begin{multline*}
(\p_t+\p_x\Delta)r_{\theta,m} = \frac{\theta}{2}\cos(\frac{\theta}{2}t)m^{-s}\sin(mx-3y+\varphi(m,-1)t)\\
-\frac{\theta}{2}\sin(\frac{\theta}{2}t)m^{-s}\cos(mx+3y+\varphi(m,1)t)+O_{L^2}(m^{-1-s}),
\end{multline*}
by using the definition of $\varphi$ and $R$. This proves (\ref{approx-sol}).

With this estimate at hand, we can control the difference between $u_{\theta,m}$ and the genuine solution arising from $u_{\theta,m}(0)$ (provided by Theorem~\ref{local-zk}) via a standard energy estimate. Let us write $u$ for this solution. First, we have that $u$ exists on the whole time interval $[0,1]$. Indeed, in Theorem~\ref{local-zk}, it is not hard to keep track of the size of $T=T(\|u(0)\|_{H^s})$ : in view of (\ref{apriori2-zk}), we can take ${\displaystyle T\sim (1+\|u(0)\|_{H^s})^{-\frac12}}$, thus $T\sim 1$ for $u(0)=u_{\theta,m}(0)$. Next, let us define $v = u - u_{\theta,m}$, then $v$ solves
\[(\p_t +\p_x\Delta)v +v\p_x v +\p_x(u_{\theta,m}v) + G = 0, \]
where $G$ is the term in (\ref{approx-sol}). Proceeding then as in Lemma~\ref{lem2-zk} at the $L^2$ level, we get
\[\frac{d}{dt}\|v(t)\|_{L^2}^2\lesssim \left(\|u_{\theta,m}(t)\|_{L^{\infty}}+\|\p_xu_{\theta,m}(t)\|_{L^{\infty}}\right)\|v(t)\|_{L^2}^2 + \|v(t)\|_{L^2}\|G(t)\|_{L^2},\]
which, along with the definition of $u_{\theta,m}$ and (\ref{approx-sol}) and that $v(0)=0$, leads to the bound
\[\sup_{t\in[0,1]}\|v(t)\|_{L^2}\lesssim m^{(1-2s)\vee (-1-s)}.\]
Interpolating with the trivial bound
\[\|v(t)\|_{H^{s+1}}\lesssim \|u(t)\|_{H^{s+1}}+\|u_{\theta,m}(t)\|_{H^{s+1}}\lesssim m,\]
which is a consequence of Lemmata~\ref{lem2-zk} and~\ref{lem4-zk} and the definition of $u_{\theta,m}$, we get at last
\begin{equation}\label{bound-th2}
\|v(t)\|_{H^s}\lesssim m^{-\varepsilon},
\end{equation}
 for some $\varepsilon>0$.

Finally, the proof of Theorem~\ref{lemma1.1} is completed as follows : take $u_m$ (respectively $v_m$) to be the solution of the IVP (\ref{zk-1}) with initial condition $u_m(0,x,y) = u_{1,m}(0,x,y)$ (respectively $v_m(0,x,y) = u_{-1,m}(0,x,y)$) given by Theorem~\ref{local-zk} (which are well-defined on the time interval $[0,1]$ as before), then (\ref{unif_bound}) is again a consequence of Lemmata~\ref{lem2-zk} and~\ref{lem4-zk}, whereas (\ref{init_bound}) is straightforward from the definition of $u_{\theta,m}(0)$. To prove (\ref{failure_bound}), we use (\ref{bound-th2}) to get the lower bound
\[\|u_m(t)-v_m(t)\|_{H^s}\geq \|u_{1,m}(t)-u_{-1,m}(t)\|_{H^s}+O(m^{-\varepsilon}).\]
 From the definition of the approximate solution and with the use of (\ref{eq-curv-reson}), we finally get
 \[\|u_{1,m}(t)-u_{-1,m}(t)\|_{H^s}=2\left|\sin\left(\frac{t}{2}\right)\right|\|\sin(mx+y+\varphi(m,1)t)\|_{L^2}+O(m^{-1})\gtrsim |t|+O(m^{-1}).\]
This completes the proof of Theorem~\ref{lemma1.1}.
\end{proof}
\begin{remark}\label{FIN}~
\begin{enumerate}[label=(\alph*)]
\item It is not difficult to see that $u_{\theta,m}$ is still a sufficiently good approximate solution even for lower values of $s$ : one can check that the last bound holds with $\varepsilon>0$ for any $s>1/2$. However, since there exists no flow map defined on $H^s(\T^2)$ when $1/2<s\leq 5/3$, we restricted our construction to those $s>5/3$ treated by Theorem~\ref{local-zk}.
\item The construction above is independent of the choice of the periods in $x$ and $y$, thus one can repeat the argument for any torus $\T^2_{\lambda} = (\R/2\pi\lambda_1\Z)\times(\R/2\pi\lambda_2\Z)$ and recover the quasi-linear behaviour likewise. In particular, there is another interesting resonant interaction~:~$R(m,m_1,\sqrt{3}m,-\sqrt{3}m_1)=0$. To perform the same construction as above, this requires to work on a torus $\T^2_{\lambda}$ whose periods $\lambda=(\lambda_1,\lambda_2)$ satisfy the condition ${\displaystyle \sqrt{3}\frac{\lambda_2}{\lambda_1}\in\mathbb{Q}}$.
\item The frequency $m_1=0$ in the aforementioned resonant interaction corresponds to the $x$-mean value ${\displaystyle \int_{\T}u(t,x,y)\,dx}$, which is preserved by the flow :
\[\int_{\T} u(t,x,y)\,dx = \int_{\T}u_0(x,y)\,dx,~\forall (t,y)\in\R\times\T.\]
We do not know the existence of a gauge transformation allowing to get rid of the contribution of this frequency. This is in contrast with the periodic KP type equations (see \emph{e.g.} \cite{JB_kp,Robert2017}), where ${\displaystyle \int_{\T}u_0(x,y)\,dx}$ must be independent of $y$ to define the anti-derivative, and where the Galilean transform reduces the problem to initial data with this constant being equal to zero.
\item At last, let us notice that the above construction uses in a crucial way the \lq\lq curvature\rq\rq property (\ref{eq-curv-reson}), meaning that for a given resonant interaction $R(m,m_1,n,n_1)=0$ we need $R(m, m+ m_1, n, n+ n_1)$ to be sufficiently large.
\end{enumerate}
\end{remark}

\end{document}